\documentclass[a4paper,11pt]{amsart}
\usepackage[colorlinks,linkcolor=blue,citecolor=blue]{hyperref}
\usepackage{latexsym, amssymb, amsmath, amsthm, mathrsfs, bbm}
\usepackage[all, knot]{xy}
\xyoption{arc}

\usepackage{anysize}\marginsize{30mm}{30mm}{30mm}{30mm}
\def \To{\longrightarrow}
\def \Hom{\operatorname{Hom}}
\def \Vec{\operatorname{Vec}}
\def \Ker{\operatorname{Ker}}
\def \H{\operatorname{H}}
\def \id{\operatorname{id}}

\def \op{\operatorname{op}}
\def \R{\mathcal{R}}

\def \k{\mathbbm{k}}
\def \o{\omega}

\numberwithin{equation}{section}

\newtheorem{theorem}{Theorem}[section]
\newtheorem{lemma}[theorem]{Lemma}
\newtheorem{proposition}[theorem]{Proposition}
\newtheorem{corollary}[theorem]{Corollary}

\begin{document}

\title{ON BRAIDED LINEAR GR-CATEGORIES$^\dag$}\thanks{\tiny $^\dag$Supported by PCSIRT IRT1264, SRFDP 20130131110001, NSFC 11371186 and SDNSF ZR2013AM022.}


\author{Hua-Lin Huang}
\address{School of Mathematics, Shandong University, Jinan 250100, China} \email{hualin@sdu.edu.cn}

\author{Gongxiang Liu}
\address{Department of Mathematics, Nanjing University, Nanjing 210093, China} \email{gxliu@nju.edu.cn}

\author{Yu Ye}
\address{School of Mathematics, USTC, Hefei 230026, China} \email{yeyu@ustc.edu.cn}

\date{}
\maketitle

\begin{abstract}
We provide explicit and unified formulae for the normalized 3-cocycles on arbitrary finite abelian groups. As an application, we compute all the braided monoidal structures on linear Gr-categories.
\end{abstract}

\section{Introduction}
Throughout the paper, let $\k$ be an algebraically closed field of characteristic zero. By $\k^*$ we denote the multiplicative group $\k-\{0\}.$ Vector spaces, morphisms and categories considered herein are over $\k$ unless otherwise specified. For unexplained notions and facts, the reader is referred to \cite{am,Wei} for group cohomology and to \cite{js,kassel} for tensor categories.

Let $G$ be a group. By a linear Gr-category over $G$ we mean a tensor category $\Vec_G^\o$ consists of finite-dimensional vector spaces graded by $G$ with the usual tensor product and with associativity constraint given by a 3-cocycle $\o$ on $G.$ The notion``Gr-category" goes back to  Ho\`{a}ng Xu\^{a}n S\'{i}nh, a student of Grothendieck. In her thesis \cite{grc}, the monoidal structures of $\Vec_G$ (the category of $G$-graded spaces) were first related to the third cohomology group of $G.$ Gr-categories are a typical class of fusion categories and, in particular, any pointed fusion category has the form $\Vec_G^\o$ and so does the full subcategory of semi-simple objects of any finite pointed tensor category; see \cite{eno, eo}.

It is well known that the monoidal structures of $\Vec_G,$ up to tensor equivalence, are parameterized by the third cohomology group $\H^3(G,\k^*)$ and $\Vec_G^\o$ is braided only if $G$ is abelian and the braidings are given by quasi-bicharacters with respect to $\o;$ see, e.g., \cite{grc, js}. However, for the braided monoidal structures on $\Vec_G$ we need explicit and unified formulae of the normalized 3-cocycles on $G.$ Though the cohomology group of a finite abelian group may be known (see, for instance, \cite{am,Wei}), the formulae of normalized cocycles (even in low degrees) seem not available in the literature, except for several easy cases. The case of cyclic groups is given in \cite{js, ms}. The first result for non-cyclic group is obtained in \cite{bct} for the Klein group via very complicated computation of happy 3-cocycles which seems not applicable to more general groups. In \cite{bgrc1}, we are able to handle the case of the direct product of any two cyclic groups. The crux is to construct a chain map from the normalized bar resolution to the tensor product of the minimal resolutions of the cyclic factors of the considered group. It turns out that this idea can be extended to general finite abelian groups, only much more delicate constructions get involved.

Aside from the obvious importance in cohomology and representation theory of groups, for examples the normalized 2-cocycles in projective representations and Schur multipliers, the explicit and unified formulae of normalized 3-cocycles are indispensable in the classification program of finite pointed tensor categories and quasi-quantum groups which recently is under intensive study; see \cite{qha1,qha2,qha3} and related works. It is also fairly reasonable to expect that our results will be useful in the studies of twisted quantum doubles \cite{dcr}, finite group modular data \cite{cgr}, group-theoretical fusion categories \cite{eno}, and in the computation of Dijkgraaf-Witten invariants \cite{dw}. We hope to treat these subjects matter in future works.

The paper is organized as follows. In Section 2, for a given finite abelian group $G$ we consider the tensor resolution of the minimal resolutions of its cyclic factors and construct a chain map, up to the third term, from the normalized bar resolution to the tensor resolution. With a help of this chain map, in Section 3 we provide an explicit and unified formulae for the normalized 3-cocycles on $G$ and compute the monoidal structures on $\Vec_G.$ Finally in Section 4, we compute the braided structures on $\Vec_G^\o.$ Our results extend those very special cases obtained in \cite{js,bct,bgrc1} to the full generality.

\section{The tensor resolution and a chain map}
Let $G$ be a group and $(B_{\bullet},\partial_{\bullet})$ its normalized bar resolution. Applying $\Hom_{\mathbb{Z}G}(-,\k^{\ast})$ one gets a complex $(B^{\ast}_{\bullet},\partial^{\ast}_{\bullet}).$ Denote the group of normalized $n$-cocycles by $Z^n(G,\k^*),$ which is $\Ker \partial^*_n.$ In general, it is hard to determine $Z^n(G,\k^*)$ directly as the normalized bar resolution is over too huge. Our strategy of overcoming this is to get first a simpler resolution of $G$ whose cocycles are easy to compute and then construct a chain map from the bar resolution to it which will help to determine $Z^n(G,\k^*)$ eventially.

\subsection{The tensor resolution.} From now on let $G$ be a finite abelian group. Write $G\cong \mathbb{Z}_{m_{1}}\times\cdots \times\mathbb{Z}_{m_{n}}$ and for every $\mathbb{Z}_{m_{i}}$ fix a generator $g_{i}$ for $1\leq i\leq n.$ It is well known that the following periodic sequence is a free resolution of the
trivial $\mathbb{Z}_{m_{i}}$-module $\mathbb{Z}:$
\begin{equation}\cdots\longrightarrow \mathbb{Z}\mathbb{Z}_{m_{i}}\stackrel{T_{i}}\longrightarrow
\mathbb{Z}\mathbb{Z}_{m_{i}}\stackrel{N_{i}}\longrightarrow\mathbb{Z}\mathbb{Z}_{m_{i}}\stackrel{T_{i}}\longrightarrow
\mathbb{Z}\mathbb{Z}_{m_{i}}\stackrel{N_{i}}\longrightarrow
\mathbb{Z}\longrightarrow 0,\end{equation}
where $T_{i}=g_{i}-1$ and $N_{i}=\sum_{j=0}^{m_{i}-1}g_{i}^{j}$.

We construct the tensor product of the above periodic resolutions of the cyclic factors of $G.$ Specifically, let
$K_{\bullet}$ be the following complex of free $\mathbb{Z}G$-modules. For
each sequence  $a_{1},\ldots,a_{n}$ of nonnegative integers, let $\Phi(a_{1},\ldots,a_{n})$ be a free
generator in degree $a_{1}+\cdots+a_{n}$. Thus
$$K_{m}:=\bigoplus_{a_{1}+\cdots+a_{n}=m} (\mathbb{Z}G)\Phi(a_{1},\ldots,a_{n}).$$ Define
$$d_{i}(\Phi(a_{1},\ldots,a_{n}))=\left \{
\begin{array}{lll} 0 &\;\;\;\;a_{i}=0
\\ (-1)^{\sum_{l<i}a_{l}}N_{i}\Phi(a_{1},\ldots,a_{i}-1,\ldots,a_{n}) & \;\;\;\;0\neq a_{i}\;\textrm{even}
\\(-1)^{\sum_{l<i}a_{l}}T_{i}\Phi(a_{1},\ldots,a_{i}-1,\ldots,a_{n}) &
\;\;\;\;0\neq a_{i}\;\textrm{odd}
\end{array} \right.$$
for $1\leq i\leq n$ and the differential $d$ is set to be $d_{1}+\cdots +d_{n}$.

\begin{lemma}
$(K_{\bullet},d_{\bullet})$ is a free resolution of the trivial $\mathbb{Z}G$-module $\mathbb{Z}$.
\end{lemma}
\begin{proof} By observing that $(K_{\bullet},d_{\bullet})$ is exactly the tensor product complex of (2.1), the lemma follows by the K\"unneth formula for complexes; see \cite[Theorem 3.6.3]{Wei}.
\end{proof}

For the convenience of the exposition, we fix some notations before moving on. For any $1\leq i\leq n$, define $\Phi_{i}:=\Phi(0,\ldots,1,\ldots,0)$ where $1$ lies in the $i$-th position.
For any $1\leq i\leq j\leq n$, define $\Phi_{i,j}:=\Phi(0,\ldots,1,\ldots,1,\ldots,0)$ where $1$ lies in the $i$-th and $j$-th positions if $i<j$ and $\Phi_{i,i}:=\Phi(0,\ldots,2,\ldots,0)$ where $2$ lies in the $i$-th position. Similarly, we define $\Phi_{i,j,k},\Phi_{i,j,j},\Phi_{i,i,j}$ and
 $\Phi_{i,i,i}$ for $1\leq i\leq n$, $1\leq i<j\leq n$ and $1\leq i<j<k\leq n$ and define $\Phi_{i,j,s,t},\Phi_{i,i,j,s},\Phi_{i,j,s,s},\Phi_{i,j,j,s},\Phi_{i,i,j,j},\Phi_{i,i,i,j},\Phi_{i,j,j,j}$
 and $\Phi_{i,i,i,i}$ for $1\leq i\leq n$, $1\leq i<j\leq n$, $1\leq i<j<s\leq n$ and $1\leq i<j<s<t\leq n$.

It is clear that any cochain $f\in \Hom_{\mathbb{Z}G}(K_{3},\k^{\ast})$ is uniquely determined by its values on $\Phi_{i,j,k},\Phi_{i,j,j},\Phi_{i,i,j}$ and $\Phi_{i,i,i}$ for $1\leq i\leq n$, $1\leq i<j\leq n$ and $1\leq i<j<k\leq n$. Write $f_{i,j,k}=f(\Phi_{i,j,k}), \ f_{i,j,j}=f(\Phi_{i,j,j}), \ f_{i,i,j}=f(\Phi_{i,i,j})$ and $f_{i,i,i}=f(\Phi_{i,i,i})$.

\begin{lemma} The $3$-cochain $f\in \Hom_{\mathbb{Z}G}(K_{3},\k^{\ast})$ is a cocycle if and only if
\begin{equation} f_{i,i,i}^{m_{i}}=1,\;\;f_{i,j,j}^{m_{i}}f_{i,i,j}^{m_{j}}=1,\;\;f_{i,j,k}^{m_{i}}=f_{i,j,k}^{m_{j}}=f_{i,j,k}^{m_{k}}=1
\end{equation}
for all $1\leq i\leq n$, $1\leq i<j\leq n$ and $1\leq i<j<k\leq n.$
\end{lemma}
\begin{proof}
By definition, $f$ is a $3$-cocycle if and only if $1=d^{\ast}(f)(\Phi_{i,j,s,t})=f(d(\Phi_{i,j,s,t}))$ for all  $1\leq i\leq j\leq s\leq t\leq n.$ For any $a\in \k^{\ast}$, clearly $T_{i}\cdot a=1$ since $\k^{\ast}$ is a trivial $G$-module. This implies that we only need to consider the condition $1=d^{\ast}(f)(\Phi_{i,j,s,t})$ in case $i=j=s=t$, $i=j<s<t$, $i<j=s<t$, $i<j<s=t$ and $i=j<s=t$. In case $i=j=s=t$, we have $1=d^{\ast}(f)(\Phi_{i,i,i,i})=f(N_{i}\Phi_{i,i,i})=N_{i}\cdot f_{i,i,i}=f_{i,i,i}^{m_{i}}$.
Similarly, in case  $i=j<s<t$, we have $f_{i,s,t}^{m_{i}}=1$. In case $i<j=s<t$, we have $f_{i,j,t}^{-m_{j}}=1$. In case $i<j<s=t$, we have $f_{i,j,s}^{m_{s}}=1$. And in case $i=j<s=t$, we have
$f_{i,s,s}^{m_{i}}f_{i,i,s}^{m_{s}}=1$.
\end{proof}

\begin{lemma}The $3$-cochain $f \in \Hom_{\mathbb{Z}G}(K_{3},\k^{\ast})$ is a coboundary if and only if for all $1\leq i<j\leq n$, there are $g_{i,j}\in \k^{\ast}$ such that
\begin{equation} f_{l,l,l}=1,\;\;f_{i,i,j}=g_{i,j}^{m_{i}},\;\;f_{i,j,j}=g_{i,j}^{-m_{j}},\;\;f_{r,s,t}=1
\end{equation} for $1\leq l\leq n$ and $1\leq r<s<t\leq n$.
\end{lemma}
\begin{proof} By definition, $f$ is a coboundary if and only if $f=d^{\ast}(g)$ for some $2$-cochain $g\in \Hom_{\mathbb{Z}G}(K_{2},\k^{\ast})$. For any $1\leq i\leq j\leq n$,
let $g_{i,j}:=g(\Phi_{i,j})$. As before, we have $T_{l}\cdot a=1$ for any $a\in \k^{\ast},$ and $d^{\ast}(g)(\Phi_{r,s,t})=d^{\ast}(g)(\Phi_{l,l,l})=1$ for $1\leq r<s<t\leq n$ and
$1\leq l\leq n.$ While for all $1\leq i<j\leq n,$ \ $f_{i,i,j}=d^{\ast}(g)(\Phi_{i,i,j})=g(N_{i}\Phi_{i,j}+T_{j}\Phi_{i,i})=g_{i,j}^{m_{i}}$ and $f_{i,j,j}=d^{\ast}(g)(\Phi_{i,j,j})=g(T_{i}\Phi_{j,j}-N_{j}\Phi_{i,j})=g_{i,j}^{-m_{j}}$.
\end{proof}

For a set of natural numbers $s_{1},\ldots,s_{t},$ by $(s_{1},\ldots,s_{t})$ we denote their greatest common divisor.

\begin{proposition}
$$\H^{3}(G,\k^{\ast})\cong \prod_{i=1}^{n}\mathbb{Z}_{m_{i}}\times \prod_{1\leq i<j\leq n}^{n}\mathbb{Z}_{(m_{i},m_{j})}\times \prod_{1\leq i<j<k\leq n}^{n}\mathbb{Z}_{(m_{i},m_{j},m_{k})}.$$
\end{proposition}
\begin{proof} By Lemmas 2.2 and 2.3, for a $3$-cocycle $f$ one can assume that $f_{l,l,l}$ is an $m_{l}$-th root of unity and $f_{i,j,k}$ is an $(m_{i},m_{j},m_{k})$-th root of unity for all $1\leq l\leq n$ and $1\leq i<j<k\leq n$. By Lemma 2.3, one can take $g_{i,j}=f_{i,j,j}^{-\frac{1}{m_{j}}}$ and thus one can assume that $f_{i,j,j}=1$
and $g_{i,j}^{m_{j}}=1$ for all $1\leq i<j \leq n$. By $f_{i,j,j}^{m_{i}}f_{i,i,j}^{m_{j}}=1$, one has $f_{i,i,j}^{m_{j}}=1$. Therefore, $\H^{3}(G,\k^{\ast})$ must be a quotient group of $\prod_{i=1}^{n}\mathbb{Z}_{m_{i}}\times \prod_{1\leq i<j\leq n}^{n}\mathbb{Z}_{m_{j}}\times \prod_{1\leq i<j<k\leq n}^{n}\mathbb{Z}_{(m_{i},m_{j},m_{k})}.$ Using the second relation in (2.3), one may even assume that $f_{i,i,j}^{m_{i}}=1$. So the proposition is proved.
\end{proof}

For any natural number $m,$ once and for all we fix $\zeta_{m}$ to be a primitive $m$-th root of unity.

\begin{corollary}
The following \[ \left\{ f\in \Hom_{\mathbb{Z}G}(K_{3},\k^{\ast})
\left|
  \begin{array}{ll}
    f_{l,l,l}=\zeta_{m_{l}}^{a_{l}},f_{i,i,j}=\zeta_{m_{j}}^{a_{ij}}, \ f_{i,j,j}=1, \
    f_{r,s,t}=\zeta_{(m_{r},m_{s},m_{t})}^{a_{rst}}\\
   \emph{for} \ 1\leq l\leq n, \ 1\leq i<j\leq n, \ 1\leq r<s<t\leq n, \ \emph{and} \\
    0 \leq a_{l}<m_{l}, \ 0\leq a_{ij}<(m_{i},m_{j}), \ 0\leq a_{rst}<(m_{r},m_{s},m_{t})
  \end{array}
\right. \right\} \]
makes a complete set of representatives of $3$-cocycles of the complex $(K_{\bullet}^*,d_{\bullet}^*).$
\end{corollary}

\subsection{A chain map.}
We need some more notations to present the chain map. For any positive integers $s$ and $t,$ let $[\frac{s}{t}]$ denote the integer part of $\frac{s}{t}$ and let $s_t'$ denote the remainder of division of $s$ by $t.$ When there is no risk of confusion, we drop the subscript and write simply $s'.$ The following
observation is useful in later arguments.

\begin{lemma} For any three natural numbers $s,t,r$, one has
\begin{equation}[\frac{s+t_r'}{r}]=[\frac{s+t}{r}]-[\frac{t}{r}].
\end{equation}
\end{lemma}

\begin{proof}
$ [\frac{s+t_r'}{r}]=[\frac{s+t-[\frac{t}{r}]r}{r}]=[\frac{s+t}{r}]-[\frac{t}{r}]. $
\end{proof}

Now we are ready to give a chain map, up to the third term for our purpose, from the normalized bar resolution $(B_{\bullet},\partial_{\bullet})$ to the tensor resolution $(K_{\bullet},d_{\bullet}).$ Recall that $B_{m}$ is the free $\mathbb{Z}G$-module on the set of all symbols
$[h_{1},\ldots,h_{m}]$ with $h_{i}\in G$ and $m\geq 1$. In case $m=0$, the symbol $[\; ]$ denote $1\in \mathbb{Z}G$ and the map $\partial_{0}=\epsilon:\;B_{0}\to \mathbb{Z}$ sends
$[\; ]$ to $1$.

We define the following three  morphisms of $\mathbb{Z}G$-modules:
\begin{eqnarray*}
F_{1}: &&B_{1}\To K_{1}\\
&&[g_{1}^{i_{1}}\cdots g_{n}^{i_{n}}]\mapsto
\sum_{s=1}^{n}\sum_{\alpha_{s}=0}^{i_{s}-1}g_{1}^{i_{1}}\cdots g_{s-1}^{i_{s-1}}g_{s}^{\alpha_{s}}\Phi_{s};\\
F_{2}: &&B_{2}\To K_{2}\\
&&[g_{1}^{i_{1}}\cdots g_{n}^{i_{n}},g_{1}^{j_{1}}\cdots g_{n}^{j_{n}}]\mapsto
\sum_{s=1}^{n}g_{1}^{i_{1}+j_{1}}\cdots g_{s-1}^{i_{s-1}+j_{s-1}}[\frac{i_{s}+j_{s}}{m_{s}}]\Phi_{s,s}\\
&&-\sum_{1\leq s<t\leq n}\sum_{\alpha_{s}=0}^{j_{s}-1}\sum_{\beta_{t}=0}^{i_{t}-1}g_{1}^{i_{1}}\cdots g_{t-1}^{i_{t-1}}g_{1}^{j_{1}}\cdots g_{s-1}^{j_{s-1}}g_{s}^{\alpha_{s}}g_{t}^{\beta_{t}}\Phi_{s,t};\\
F_{3}: &&B_{3}\To K_{3}\\
&&[g_{1}^{i_{1}}\cdots g_{n}^{i_{n}},g_{1}^{j_{1}}\cdots g_{n}^{j_{n}},g_{1}^{k_{1}}\cdots g_{n}^{k_{n}}]\mapsto\\
&& \sum_{r=1}^{n}[\frac{j_{r}+k_{r}}{m_{r}}]g_{1}^{j_{1}+k_{1}}\cdots g_{r-1}^{j_{r-1}+k_{r-1}}\sum_{\beta_{r}=0}^{i_{r}-1}g_{1}^{i_{1}}\cdots g_{r-1}^{i_{r-1}}g_{r}^{\beta_{r}}\Phi_{r,r,r}+\\
&&\sum_{1\leq r<t\leq n}[\frac{j_{r}+k_{r}}{m_{r}}]g_{1}^{j_{1}+k_{1}}\cdots g_{r-1}^{j_{r-1}+k_{r-1}}
\sum_{\beta_{t}=0}^{i_{t}-1}g_{1}^{i_{1}}\cdots g_{t-1}^{i_{t-1}}g_{t}^{\beta_{t}}\Phi_{r,r,t}+\\
&&\sum_{1\leq r<t\leq n}[\frac{i_{t}+j_{t}}{m_{t}}]g_{1}^{i_{1}+j_{1}}\cdots g_{t-1}^{i_{t-1}+j_{t-1}}
\sum_{\gamma_{r}=0}^{k_{r}-1}g_{1}^{k_{1}}\cdots g_{r-1}^{k_{r-1}}g_{r}^{\gamma_{r}}\Phi_{r,t,t}-\\
&&\sum_{1\leq r<s<t\leq n}\sum_{\beta_{t}=0}^{i_{t}-1}g_{1}^{i_{1}}\cdots g_{t-1}^{i_{t-1}}g_{t}^{\beta_{t}}
\sum_{\alpha_{s}=0}^{j_{s}-1}g_{1}^{j_{1}}\cdots g_{s-1}^{j_{s-1}}g_{s}^{\alpha_{s}}
\sum_{\gamma_{r}=0}^{k_{r}-1}g_{1}^{k_{1}}\cdots g_{r-1}^{k_{r-1}}g_{r}^{\gamma_{r}}\Phi_{r,s,t}
\end{eqnarray*}
for $0\leq i_{l},j_{l},k_{l}< m_{l}$ and $1\leq l\leq n$.

\begin{proposition} The following diagram is commutative

\begin{figure}[hbt]
\begin{picture}(150,50)(50,-40)
\put(0,0){\makebox(0,0){$ \cdots$}}\put(10,0){\vector(1,0){20}}\put(40,0){\makebox(0,0){$B_{3}$}}
\put(50,0){\vector(1,0){20}}\put(80,0){\makebox(0,0){$B_{2}$}}
\put(90,0){\vector(1,0){20}}\put(120,0){\makebox(0,0){$B_{1}$}}
\put(130,0){\vector(1,0){20}}\put(160,0){\makebox(0,0){$B_{0}$}}
\put(170,0){\vector(1,0){20}}\put(200,0){\makebox(0,0){$\mathbb{Z}$}}
\put(210,0){\vector(1,0){20}}\put(240,0){\makebox(0,0){$0$}}

\put(0,-40){\makebox(0,0){$ \cdots$}}\put(10,-40){\vector(1,0){20}}\put(40,-40){\makebox(0,0){$K_{3}$}}
\put(50,-40){\vector(1,0){20}}\put(80,-40){\makebox(0,0){$K_{2}$}}
\put(90,-40){\vector(1,0){20}}\put(120,-40){\makebox(0,0){$K_{1}$}}
\put(130,-40){\vector(1,0){20}}\put(160,-40){\makebox(0,0){$K_{0}$}}
\put(170,-40){\vector(1,0){20}}\put(200,-40){\makebox(0,0){$\mathbb{Z}$}}
\put(210,-40){\vector(1,0){20}}\put(240,-40){\makebox(0,0){$0$}}

\put(40,-10){\vector(0,-1){20}}
\put(80,-10){\vector(0,-1){20}}
\put(120,-10){\vector(0,-1){20}}
\put(158,-10){\line(0,-1){20}}\put(160,-10){\line(0,-1){20}}
\put(198,-10){\line(0,-1){20}}\put(200,-10){\line(0,-1){20}}

\put(60,5){\makebox(0,0){$\partial_{3}$}}
\put(100,5){\makebox(0,0){$\partial_{2}$}}
\put(140,5){\makebox(0,0){$\partial_{1}$}}

\put(60,-35){\makebox(0,0){$d$}}
\put(100,-35){\makebox(0,0){$d$}}
\put(140,-35){\makebox(0,0){$d$}}

\put(50,-20){\makebox(0,0){$F_{3}$}}
\put(90,-20){\makebox(0,0){$F_{2}$}}
\put(130,-20){\makebox(0,0){$F_{1}$}}

\end{picture}
\end{figure}
\end{proposition}

\begin{proof} The proof is by direct but very complicated computation. The essence of the proposition lies in figuring out the morphisms $F_{1},F_{2}$ and $F_{3}$ in the first place. We hope that the proof may shed some light on the construction of them. The proof is naturally divided into three parts.

\textbf{Claim 1: $dF_{1}=\partial_{1}$. } Take any generator $[g_{1}^{i_{1}}\cdots g_{n}^{i_{n}}]\in B_{1}$, then
$\partial_{1}([g_{1}^{i_{1}}\cdots g_{n}^{i_{n}}])=(g_{1}^{i_{1}}\cdots g_{n}^{i_{n}}-1)\Phi(0,\ldots,0)$.  And,
 \begin{eqnarray*}
 dF_{1}([g_{1}^{i_{1}}\cdots g_{n}^{i_{n}}])&=&d(\sum_{s=1}^{n}\sum_{\alpha_{s}=0}^{i_{s}-1}g_{1}^{i_{1}}\cdots g_{s-1}^{i_{s-1}}g_{s}^{\alpha_{s}}\Phi_{s})\\
&=&\sum_{s=1}^{n}\sum_{\alpha_{s}=0}^{i_{s}-1}g_{1}^{i_{1}}\cdots g_{s-1}^{i_{s-1}}g_{s}^{\alpha_{s}}(g_{s}-1)\Phi(0,\ldots,0)\\
&=&\sum_{s=1}^{n}g_{1}^{i_{1}}\cdots g_{s-1}^{i_{s-1}}(g_{s}^{i_{s}}-1)\Phi(0,\ldots,0)\\
&=&(g_{1}^{i_{1}}\cdots g_{n}^{i_{n}}-1)\Phi(0,\ldots,0).
\end{eqnarray*}

\textbf{Claim 2: $dF_{2}=F_{1}\partial_{2}$. } For any generator $[g_{1}^{i_{1}}\cdots g_{n}^{i_{n}},g_{1}^{j_{1}}\cdots g_{n}^{j_{n}}]$, we have
  \begin{eqnarray*}&&F_{1}\partial_{2}([g_{1}^{i_{1}}\cdots g_{n}^{i_{n}},g_{1}^{j_{1}}\cdots g_{n}^{j_{n}}])\\
  &=&F_{1}(g_{1}^{i_{1}}\cdots g_{n}^{i_{n}}[g_{1}^{j_{1}}\cdots g_{n}^{j_{n}}]-[g_{1}^{i_{1}+j_{1}}\cdots g_{n}^{i_{n}+j_{n}}]+[g_{1}^{i_{1}}\cdots g_{n}^{i_{n}}])\\
  &=&g_{1}^{i_{1}}\cdots g_{n}^{i_{n}}\sum_{s=1}^{n}\sum_{\alpha_{s}=0}^{j_{s}-1}g_{1}^{j_{1}}\cdots g_{s-1}^{j_{s-1}}g_{s}^{\alpha_{s}}\Phi_{s}\\
  &&-\sum_{s=1}^{n}\sum_{\alpha_{s}=0}^{(i_{s}+j_{s})'-1}g_{1}^{i_{1}+j_{1}}\cdots g_{s-1}^{i_{s-1}+j_{s-1}}g_{s}^{\alpha_{s}}\Phi_{s}\\
  &&+\sum_{s=1}^{n}\sum_{\alpha_{s}=0}^{i_{s}-1}g_{1}^{i_{1}}\cdots g_{s-1}^{i_{s-1}}g_{s}^{\alpha_{s}}\Phi_{s}.
   \end{eqnarray*}
   Fix any $s$, the coefficient of $\Phi_{s}$ is
   \begin{eqnarray}&&g_{1}^{i_{1}}\cdots g_{n}^{i_{n}}\sum_{\alpha_{s}=0}^{j_{s}-1}g_{1}^{j_{1}}\cdots g_{s-1}^{j_{s-1}}g_{s}^{\alpha_{s}} \notag \\
   &&-g_{1}^{i_{1}+j_{1}}\cdots g_{s-1}^{i_{s-1}+j_{s-1}}(\sum_{\alpha_{s}=0}^{i_{s}+j_{s}-1}g_{s}^{\alpha_{s}}-[\frac{i_{s}+j_{s}}{m_{s}}]N_{s})\\\notag
   &&+\sum_{\alpha_{s}=0}^{i_{s}-1}g_{1}^{i_{1}}\cdots g_{s-1}^{i_{s-1}}g_{s}^{\alpha_{s}}.\end{eqnarray}

Now consider $dF_{2}.$ We have
   \begin{eqnarray*} &&dF_{2}([g_{1}^{i_{1}}\cdots g_{n}^{i_{n}},g_{1}^{j_{1}}\cdots g_{n}^{j_{n}}]) \\ &=&d(\sum_{s=1}^{n}g_{1}^{i_{1}+j_{1}} \cdots g_{s-1}^{i_{s-1}+j_{s-1}}[\frac{i_{s}+j_{s}}{m_{s}}]\Phi_{s,s}) \\ &&-d(\sum_{1\leq s<t\leq n}\sum_{\alpha_{s}=0}^{j_{s}-1}\sum_{\beta_{t}=0}^{i_{t}-1}g_{1}^{i_{1}}\cdots g_{t-1}^{i_{t-1}}g_{1}^{j_{1}}\cdots g_{s-1}^{j_{s-1}}g_{s}^{\alpha_{s}}g_{t}^{\beta_{t}}\Phi_{s,t}).
   \end{eqnarray*}
In this expression, the coefficient of $\Phi_{s}$ is
   \begin{eqnarray*}&&g_{1}^{i_{1}+j_{1}}\cdots g_{s-1}^{i_{s-1}+j_{s-1}}[\frac{i_{s}+j_{s}}{m_{s}}]N_{s}\\
   &&-\sum_{1\leq t<s}\sum_{\beta_{s}=0}^{i_{s}-1}g_{1}^{i_{1}}\cdots g_{s-1}^{i_{s-1}}g_{1}^{j_{1}}\cdots g_{t-1}^{j_{t-1}}(g_{t}^{j_{t}}-1)g_{s}^{\beta_{s}}\\
   &&+\sum_{s<t\leq n}\sum_{\alpha_{s}=0}^{j_{s}-1}g_{1}^{i_{1}}\cdots g_{t-1}^{i_{t-1}}g_{1}^{j_{1}}\cdots g_{s-1}^{j_{s-1}}g_{s}^{\alpha_{s}}(g_{t}^{i_{t}}-1)\\
   &=&g_{1}^{i_{1}+j_{1}}\cdots g_{s-1}^{i_{s-1}+j_{s-1}}[\frac{i_{s}+j_{s}}{m_{s}}]N_{s}\\&&
   -\sum_{\beta_{s}=0}^{i_{s}-1}g_{1}^{i_{1}}\cdots g_{s-1}^{i_{s-1}}(g_{1}^{j_{1}}\cdots g_{s-1}^{j_{s-1}}-1)g_{s}^{\beta_{s}}\\
   &&+\sum_{\alpha_{s}=0}^{j_{s}-1}(g_{1}^{i_{1}}\cdots g_{n}^{i_{n}}-g_{1}^{i_{1}}\cdots g_{s}^{i_{s}})g_{1}^{j_{1}}\cdots g_{s-1}^{j_{s-1}}g_{s}^{\alpha_{s}}\\
   &=& g_{1}^{i_{1}+j_{1}}\cdots g_{s-1}^{i_{s-1}+j_{s-1}}[\frac{i_{s}+j_{s}}{m_{s}}]N_{s}\\&&
   -\sum_{\beta_{s}=0}^{i_{s}+j_{s}-1}g_{1}^{i_{1}+j_{1}}\cdots g_{s-1}^{i_{s-1}+j_{s-1}}g_{s}^{\beta_{s}}\\
   &&+\sum_{\alpha_{s}=0}^{j_{s}-1}g_{1}^{i_{1}}\cdots g_{n}^{i_{n}}g_{1}^{j_{1}}\cdots g_{s-1}^{j_{s-1}}g_{s}^{\alpha_{s}}+\sum_{\beta_{s}=0}^{i_{s}-1}g_{1}^{i_{1}}\cdots g_{s-1}^{i_{s-1}}g_{s}^{\beta_{s}},\end{eqnarray*}
   which is clearly identical with (2.5). So we have $dF_{2}=F_{1}\partial_{2}.$

   \textbf{Claim 3: $dF_{3}=F_{2}\partial_{3}$. } Similarly, for any generator $[g_{1}^{i_{1}}\cdots g_{n}^{i_{n}},g_{1}^{j_{1}}\cdots g_{n}^{j_{n}}, g_{1}^{k_{1}}\cdots g_{n}^{k_{n}}]$, we have
   \begin{eqnarray*}&&F_{2}\partial_{3}([g_{1}^{i_{1}}\cdots g_{n}^{i_{n}},g_{1}^{j_{1}}\cdots g_{n}^{j_{n}},g_{1}^{k_{1}}\cdots g_{n}^{k_{n}}])\\
   &=&F_{2}(g_{1}^{i_{1}}\cdots g_{n}^{i_{n}}[g_{1}^{j_{1}}\cdots g_{n}^{j_{n}},g_{1}^{k_{1}}\cdots g_{n}^{k_{n}}]-[g_{1}^{i_{1}+j_{1}}\cdots g_{n}^{i_{n}+j_{n}},g_{1}^{k_{1}}\cdots g_{n}^{k_{n}}])\\
   &&+F_{2}([g_{1}^{i_{1}}\cdots g_{n}^{i_{n}},g_{1}^{j_{1}+k_{1}}\cdots g_{n}^{j_{n}+k_{n}}]-[g_{1}^{i_{1}}\cdots g_{n}^{i_{n}},g_{1}^{j_{1}}\cdots g_{n}^{j_{n}}])\\
   &=&g_{1}^{i_{1}}\cdots g_{n}^{i_{n}}\sum_{s=1}^{n}g_{1}^{j_{1}+k_{1}}\cdots g_{s-1}^{j_{s-1}+k_{s-1}}[\frac{j_{s}+k_{s}}{m_{s}}]\Phi_{s,s}\\
   &&-g_{1}^{i_{1}}\cdots g_{n}^{i_{n}}\sum_{1\leq s<t\leq n}\sum_{\alpha_{s}=0}^{k_{s}-1}\sum_{\beta_{t}=0}^{j_{t}-1}g_{1}^{j_{1}}\cdots g_{t-1}^{j_{t-1}}g_{1}^{k_{1}}\cdots g_{s-1}^{k_{s-1}}g_{s}^{\alpha_{s}}g_{t}^{\beta_{t}}\Phi_{s,t}\\
   &&-\sum_{s=1}^{n}g_{1}^{i_{1}+j_{1}+k_{1}}\cdots g_{s-1}^{i_{s-1}+j_{s-1}+k_{s-1}}[\frac{(i_{s}+j_{s})'+k_{s}}{m_{s}}]\Phi_{s,s}\\
   &&+\sum_{1\leq s<t\leq n}\sum_{\alpha_{s}=0}^{k_{s}-1}\sum_{\beta_{t}=0}^{(i_{t}+j_{t})'-1}g_{1}^{i_{1}+j_{1}}\cdots g_{t-1}^{i_{t-1}+j_{t-1}}g_{1}^{k_{1}}\cdots g_{s-1}^{k_{s-1}}g_{s}^{\alpha_{s}}g_{t}^{\beta_{t}}\Phi_{s,t}\\
   &&+\sum_{s=1}^{n}g_{1}^{i_{1}+j_{1}+k_{1}}\cdots g_{s-1}^{i_{s-1}+j_{s-1}+k_{s-1}}[\frac{i_{s}+(j_{s}+k_{s})'}{m_{s}}]\Phi_{s,s}\\
   &&-\sum_{1\leq s<t\leq n}\sum_{\alpha_{s}=0}^{(j_{s}+k_{s})'-1}\sum_{\beta_{t}=0}^{i_{t}-1}g_{1}^{i_{1}}\cdots g_{t-1}^{i_{t-1}}g_{1}^{j_{1}+k_{1}}\cdots g_{s-1}^{j_{s-1}+k_{s-1}}g_{s}^{\alpha_{s}}g_{t}^{\beta_{t}}\Phi_{s,t}\\
   &&-\sum_{s=1}^{n}g_{1}^{i_{1}+j_{1}}\cdots g_{s-1}^{i_{s-1}+j_{s-1}}[\frac{i_{s}+j_{s}}{m_{s}}]\Phi_{s,s}\\
   &&+\sum_{1\leq s<t\leq n}\sum_{\alpha_{s}=0}^{j_{s}-1}\sum_{\beta_{t}=0}^{i_{t}-1}g_{1}^{i_{1}}\cdots g_{t-1}^{i_{t-1}}g_{1}^{j_{1}}\cdots g_{s-1}^{j_{s-1}}g_{s}^{\alpha_{s}}g_{t}^{\beta_{t}}\Phi_{s,t}.
   \end{eqnarray*}
Note that in $(i_s+j_s)'$ we drop the subscript $m_s.$ In the previous expression, for any $1\leq s\leq n$, the coefficient of $\Phi_{s,s}$ is
    \begin{eqnarray}&&g_{1}^{i_{1}}\cdots g_{n}^{i_{n}}g_{1}^{j_{1}+k_{1}}\cdots g_{s-1}^{j_{s-1}+k_{s-1}}[\frac{j_{s}+k_{s}}{m_{s}}] \notag \\
    &&+g_{1}^{i_{1}+j_{1}+k_{1}}\cdots g_{s-1}^{i_{s-1}+j_{s-1}+k_{s-1}}([\frac{i_{s}+j_{s}}{m_{s}}]-[\frac{j_{s}+k_{s}}{m_{s}}])\\\notag
    &&-g_{1}^{i_{1}+j_{1}}\cdots g_{s-1}^{i_{s-1}+j_{s-1}}[\frac{i_{s}+j_{s}}{m_{s}}],\end{eqnarray} where Lemma 2.6 is applied.
    For any $1\leq s< t\leq n$, the coefficient of $\Phi_{s,t}$ is
     \begin{eqnarray}&&-g_{1}^{i_{1}}\cdots g_{n}^{i_{n}}\sum_{\alpha_{s}=0}^{k_{s}-1}\sum_{\beta_{t}=0}^{j_{t}-1}g_{1}^{j_{1}}\cdots g_{t-1}^{j_{t-1}}g_{1}^{k_{1}}\cdots g_{s-1}^{k_{s-1}}g_{s}^{\alpha_{s}}g_{t}^{\beta_{t}}\notag \\
    &&+\sum_{\alpha_{s}=0}^{k_{s}-1}g_{1}^{i_{1}+j_{1}}\cdots g_{t-1}^{i_{t-1}+j_{t-1}}g_{1}^{k_{1}}\cdots g_{s-1}^{k_{s-1}}g_{s}^{\alpha_{s}}(\sum_{\beta_{t}=0}^{i_{t}+j_{t}-1}g_{t}^{\beta_{t}}
    -[\frac{i_{t}+j_{t}}{m_{t}}]N_{t})\\\notag
    &&-\sum_{\beta_{t}=0}^{i_{t}-1}g_{1}^{i_{1}}\cdots g_{t-1}^{i_{t-1}}g_{t}^{\beta_{t}}g_{1}^{j_{1}+k_{1}}\cdots g_{s-1}^{j_{s-1}+k_{s-1}}(\sum_{\alpha_{s}=0}^{j_{s}+k_{s}-1}g_{s}^{\alpha_{s}}-
    [\frac{j_{s}+k_{s}}{m_{s}}]N_{s})\\\notag
    &&+\sum_{\alpha_{s}=0}^{j_{s}-1}\sum_{\beta_{t}=0}^{i_{t}-1}g_{1}^{i_{1}}\cdots g_{t-1}^{i_{t-1}}g_{1}^{j_{1}}\cdots g_{s-1}^{j_{s-1}}g_{s}^{\alpha_{s}}g_{t}^{\beta_{t}}.\end{eqnarray}

For $dF_{3},$ we have
      \begin{eqnarray*}&&dF_{3}([g_{1}^{i_{1}}\cdots g_{n}^{i_{n}},g_{1}^{j_{1}}\cdots g_{n}^{j_{n}},g_{1}^{k_{1}}\cdots g_{n}^{k_{n}}])\\
      &=& d(\sum_{r=1}^{n}[\frac{j_{r}+k_{r}}{m_{r}}]g_{1}^{j_{1}+k_{1}}\cdots g_{r-1}^{j_{r-1}+k_{r-1}}\sum_{\beta_{r}=0}^{i_{r}-1}g_{1}^{i_{1}}\cdots g_{r-1}^{i_{r-1}}g_{r}^{\beta_{r}}\Phi_{r,r,r})+\\
      &&d(\sum_{1\leq r<t\leq n}[\frac{j_{r}+k_{r}}{m_{r}}]g_{1}^{j_{1}+k_{1}}\cdots g_{r-1}^{j_{r-1}+k_{r-1}}
      \sum_{\beta_{t}=0}^{i_{t}-1}g_{1}^{i_{1}}\cdots g_{t-1}^{i_{t-1}}g_{t}^{\beta_{t}}\Phi_{r,r,t})+\\
      &&d(\sum_{1\leq r<t\leq n}[\frac{i_{t}+j_{t}}{m_{t}}]g_{1}^{i_{1}+j_{1}}\cdots g_{t-1}^{i_{t-1}+j_{t-1}}
      \sum_{\gamma_{r}=0}^{k_{r}-1}g_{1}^{k_{1}}\cdots g_{r-1}^{k_{r-1}}g_{r}^{\gamma_{r}}\Phi_{r,t,t})-\\
      &&d(\sum_{1\leq r<s<t\leq n}\sum_{\beta_{t}=0}^{i_{t}-1}g_{1}^{i_{1}}\cdots g_{t-1}^{i_{t-1}}g_{t}^{\beta_{t}}
      \sum_{\alpha_{s}=0}^{j_{s}-1}g_{1}^{j_{1}}\cdots g_{s-1}^{j_{s-1}}g_{s}^{\alpha_{s}}
     \sum_{\gamma_{r}=0}^{k_{r}-1}g_{1}^{k_{1}}\cdots g_{r-1}^{k_{r-1}}g_{r}^{\gamma_{r}}\Phi_{r,s,t}).\\
     \end{eqnarray*}
Note that the coefficient of $\Phi_{s,s}$ is
      \begin{eqnarray*}&&[\frac{j_{s}+k_{s}}{m_{s}}]g_{1}^{j_{1}+k_{1}}\cdots g_{s-1}^{j_{s-1}+k_{s-1}}g_{1}^{i_{1}}\cdots g_{s-1}^{i_{s-1}}(g_{s}^{i_{s}}-1)\\
      &&+\sum_{ s<t\leq n}[\frac{j_{s}+k_{s}}{m_{s}}]g_{1}^{j_{1}+k_{1}}\cdots g_{r-1}^{j_{s-1}+k_{s-1}}
      g_{1}^{i_{1}}\cdots g_{t-1}^{i_{t-1}}(g_{t}^{i_{t}}-1)\\
      &&+\sum_{1\leq r<s}[\frac{i_{s}+j_{s}}{m_{s}}]g_{1}^{i_{1}+j_{1}}\cdots g_{s-1}^{i_{s-1}+j_{s-1}}
      g_{1}^{k_{1}}\cdots g_{r-1}^{k_{r-1}}(g_{r}^{k_{r}}-1)\\
      &=&[\frac{j_{s}+k_{s}}{m_{s}}]g_{1}^{j_{1}+k_{1}}\cdots g_{s-1}^{j_{s-1}+k_{s-1}}g_{1}^{i_{1}}\cdots g_{s-1}^{i_{s-1}}(g_{s}^{i_{s}}-1)\\
      &&+[\frac{j_{s}+k_{s}}{m_{s}}]g_{1}^{j_{1}+k_{1}}\cdots g_{r-1}^{j_{s-1}+k_{s-1}}
      (g_{1}^{i_{1}}\cdots g_{n}^{i_{n}}-g_{1}^{i_{1}}\cdots g_{s}^{i_{s}})\\
      &&+[\frac{i_{s}+j_{s}}{m_{s}}]g_{1}^{i_{1}+j_{1}}\cdots g_{s-1}^{i_{s-1}+j_{s-1}}
      (g_{1}^{k_{1}}\cdots g_{s-1}^{k_{s-1}}-1),
      \end{eqnarray*} which clearly is equal to (2.6).

Finally we consider the coefficient of $\Phi_{s,t}$ for $1\leq s<t\leq n$, which is
    \begin{eqnarray*}
     &&[\frac{j_{s}+k_{s}}{m_{s}}]N_{s}g_{1}^{j_{1}+k_{1}}\cdots g_{s-1}^{j_{s-1}+k_{s-1}}
      \sum_{\beta_{t}=0}^{i_{t}-1}g_{1}^{i_{1}}\cdots g_{t-1}^{i_{t-1}}g_{t}^{\beta_{t}}\\
      &&+[\frac{i_{t}+j_{t}}{m_{t}}]N_{t}g_{1}^{i_{1}+j_{1}}\cdots g_{t-1}^{i_{t-1}+j_{t-1}}
      \sum_{\gamma_{s}=0}^{k_{s}-1}g_{1}^{k_{1}}\cdots g_{s-1}^{k_{s-1}}g_{s}^{\gamma_{s}}\\
      &&-\sum_{1\leq r<s<t}\sum_{\beta_{t}=0}^{i_{t}-1}g_{1}^{i_{1}}\cdots g_{t-1}^{i_{t-1}}g_{t}^{\beta_{t}}
      \sum_{\alpha_{s}=0}^{j_{s}-1}g_{1}^{j_{1}}\cdots g_{s-1}^{j_{s-1}}g_{s}^{\alpha_{s}}
     g_{1}^{k_{1}}\cdots g_{r-1}^{k_{r-1}}(g_{r}^{k_{r}}-1)\\
     &&+\sum_{ s<r<t}\sum_{\beta_{t}=0}^{i_{t}-1}g_{1}^{i_{1}}\cdots g_{t-1}^{i_{t-1}}g_{t}^{\beta_{t}}
     g_{1}^{j_{1}}\cdots g_{r-1}^{j_{r-1}}(g_{r}^{j_{r}}-1)
     \sum_{\gamma_{s}=0}^{k_{s}-1}g_{1}^{k_{1}}\cdots g_{s-1}^{k_{s-1}}g_{s}^{\gamma_{s}}\\
     &&-\sum_{s<t<r}g_{1}^{i_{1}}\cdots g_{r-1}^{i_{r-1}}(g_{r}^{i_{r}}-1)
      \sum_{\alpha_{t}=0}^{j_{t}-1}g_{1}^{j_{1}}\cdots g_{t-1}^{j_{t-1}}g_{t}^{\alpha_{t}}
      \sum_{\gamma_{s}=0}^{k_{s}-1}g_{1}^{k_{1}}\cdots g_{s-1}^{k_{s-1}}g_{s}^{\gamma_{s}}\\
      &=&[\frac{j_{s}+k_{s}}{m_{s}}]N_{s}g_{1}^{j_{1}+k_{1}}\cdots g_{s-1}^{j_{s-1}+k_{s-1}}
      \sum_{\beta_{t}=0}^{i_{t}-1}g_{1}^{i_{1}}\cdots g_{t-1}^{i_{t-1}}g_{t}^{\beta_{t}}\\
      &&+[\frac{i_{t}+j_{t}}{m_{t}}]N_{t}g_{1}^{i_{1}+j_{1}}\cdots g_{t-1}^{i_{t-1}+j_{t-1}}
      \sum_{\gamma_{s}=0}^{k_{s}-1}g_{1}^{k_{1}}\cdots g_{s-1}^{k_{s-1}}g_{s}^{\gamma_{s}}\\
      &&-\sum_{\beta_{t}=0}^{i_{t}-1}g_{1}^{i_{1}}\cdots g_{t-1}^{i_{t-1}}g_{t}^{\beta_{t}}
      \sum_{\alpha_{s}=0}^{j_{s}-1}g_{1}^{j_{1}}\cdots g_{s-1}^{j_{s-1}}g_{s}^{\alpha_{s}}
     (g_{1}^{k_{1}}\cdots g_{s-1}^{k_{s-1}}-1)\\
     &&+\sum_{\beta_{t}=0}^{i_{t}-1}g_{1}^{i_{1}}\cdots g_{t-1}^{i_{t-1}}g_{t}^{\beta_{t}}
     (g_{1}^{j_{1}}\cdots g_{t-1}^{j_{t-1}}-g_{1}^{j_{1}}\cdots g_{s}^{j_{s}})
     \sum_{\gamma_{s}=0}^{k_{s}-1}g_{1}^{k_{1}}\cdots g_{s-1}^{k_{s-1}}g_{s}^{\gamma_{s}}\\
     &&-(g_{1}^{i_{1}}\cdots g_{n}^{i_{n}}-g_{1}^{i_{1}}\cdots g_{t}^{i_{t}})
      \sum_{\alpha_{t}=0}^{j_{t}-1}g_{1}^{j_{1}}\cdots g_{t-1}^{j_{t-1}}g_{t}^{\alpha_{t}}
      \sum_{\gamma_{s}=0}^{k_{s}-1}g_{1}^{k_{1}}\cdots g_{s-1}^{k_{s-1}}g_{s}^{\gamma_{s}}.
    \end{eqnarray*} It is not hard to see that this is equal to (2.7). Therefore, $dF_{3}=F_{2}\partial_{3}.$

The proof is completed.
\end{proof}

\section{Monoidal structures and normalized 3-cocycles}
\subsection{Monoidal structures}
Recall that the category $\Vec_G$ of finite-dimensional $G$-graded vector spaces has simple objects $\{V_g|g \in G\}$ where $(V_g)_h=\delta_{g,h}\k, \ \forall h \in G.$ The tensor product is given by $V_g \otimes V_h=V_{gh},$ and $V_1$ ($1$ is the identity of $G$) is the unit object. Without loss of generality we may assume that the left and right unit constraints are identities. If $a$ is an associativity constraint on $\Vec_G,$ then it is given by $a_{V_f,V_g,V_h}=\o(f,g,h)\id,$ where $\o:G \times G \times G \rightarrow \k^*$ is a function. The pentagon axiom and the triangle axiom give
\begin{gather*}
\o(ef,g,h)\o(e,f,gh)=\o(e,f,g)\o(e,fg,h)\o(f,g,h), \\ \o(f,1,g)=1,
\end{gather*}
which exactly say that $\o$ is a normalized 3-cocycle on $G.$ Note that cohomologous cocycles define equivalent monoidal structures, therefore the classification of monoidal structures on $\Vec_G$ is equivalent to determining a complete set of representatives of normalized 3-cocycles on $G.$

\subsection{Normalized 3-cocycles} Now we are able to accomplish the main task with a help of the results obtained in Section 2. Define $A$ to be the set of all sequences like
\begin{equation}(a_{1},\ldots,a_{l},\ldots,a_{n},a_{12},\ldots,a_{ij},\ldots,a_{n-1,n},a_{123},
\ldots,a_{rst},\ldots,a_{n-2,n-1,n})\end{equation}
such that $ 0\leq a_{l}<m_{l}, \ 0\leq a_{ij}<(m_{i},m_{j}), \ 0\leq a_{rst}<(m_{r},m_{s},m_{t})$ for $1\leq l\leq n, \ 1\leq i<j\leq n, \ 1\leq r<s<t\leq n$ where $a_{ij}$ and $a_{rst}$ are ordered by the lexicographic order. In the following, the sequence (3.1) is denoted by $\underline{\mathbf{a}}$ for short.

For any $\underline{\mathbf{a}}\in A$, define a $\mathbb{Z}G$-module morphism:
\begin{eqnarray}
&& \o_{\underline{\mathbf{a}}}:\;B_{3}\To \k^{\ast} \notag \\
&&[g_{1}^{i_{1}}\cdots g_{n}^{i_{n}},g_{1}^{j_{1}}\cdots g_{n}^{j_{n}},g_{1}^{k_{1}}\cdots g_{n}^{k_{n}}] \mapsto \\ &&\prod_{l=1}^{n}\zeta_{l}^{a_{l}i_{l}[\frac{j_{l}+k_{l}}{m_{l}}]}
\prod_{1\leq s<t\leq n}\zeta_{m_{t}}^{a_{st}i_{t}[\frac{j_{s}+k_{s}}{m_{s}}]}
\prod_{1\leq r<s<t\leq n}\zeta_{(m_{r},m_{s},m_{t})}^{-a_{rst}k_{r}j_{s}i_{t}}. \notag
\end{eqnarray}

\begin{proposition}  Suppose that $\k$ is an algebraically closed field of characteristic zero and $G=\mathbb{Z}_{m_{1}}\times\cdots \times\mathbb{Z}_{m_{n}}.$ Then $\{\o_{\underline{\mathbf{a}}}|\underline{\mathbf{a}}\in A\}$ is a complete set of representatives of normalized $3$-cocycles on $G.$
\end{proposition}
\begin{proof} This is a direct consequence of Corollary 2.5 and the definition of the map $F_{3}$ given in Proposition 2.7.
\end{proof}

\section{Braided structures on $\Vec_G^\o$}
\subsection{Braided structures} Now we consider the braided structures on a linear Gr-category $\Vec_G^\o.$ Recall that a braiding in $\Vec_G^\o$ is a commutativity constraint $c: \otimes \rightarrow \otimes^{\op}$ satisfying the hexagon axiom. Note that $c$ is given by $c_{V_x,V_y}=\R(x,y)\id,$ where $\R: G \times G \rightarrow \k^*$ is a function, and the hexagon axiom of $c$ says that
\begin{gather}
\R(xy, z)=\o(z , x , y)\R(x , z)\o^{-1}(x , z , y) \R(y ,z)\o(x , y , z),\\
\R(x , yz)=\o^{-1}(y , z , x)\R(x , z) \o(y , x , z) \R(x, y)\o^{-1}(x , y , z),
\end{gather} for all $x,y.z \in G.$
In other words, $\R$ is a quasi-bicharacter of $G$ with respect to $\o.$ Therefore, the classification of braidings in $\Vec_G^\o$ is equivalent to determining all the quasi-bicharacters of $G$ with respect to $\o.$

\subsection{Quasi-bicharacters} By Proposition 3.1, one may assume that $\o=\o_{\underline{\mathbf{a}}}$ for some $\underline{\mathbf{a}}\in A.$ Clearly, any quasi-bicharacter $\R$ is uniquely determined by the following values:
$$r_{ij}:=\R(g_{i},g_{j}),\;\;\;\;\;\;\;\;\textrm{for all}\;\;1\leq i, \ j\leq n.$$
For the convenience of the computation, we rewrite equations (4.1) and (4.2) as
\begin{gather}
\R(xy, z)=\R(x , z)\R(y ,z)\frac{\o(z,x,y)\o(x,y,z)}{\o(x,z,y)},\\
\R(x , yz)=\R(x , y)\R(x , z)\frac{\o(y,x,z)}{\o(y,z,x)\o(x,y,z)}.
\end{gather}

\begin{proposition} Let $\o=\o_{\underline{\mathbf{a}}}$ for some $\underline{\mathbf{a}}\in A$  and $r_{ij}\in \k^{\ast}$ for
$1\leq i, \ j\leq n$. Then there is a quasi-bicharacter $\R$ with respect to $\o_{\underline{\mathbf{a}}}$ satisfying $\R(g_{i},g_{j})=r_{ij}$
if and only if the following equations are satisfied:
\begin{gather*}r_{ii}^{m_{i}}=\zeta_{m_{i}}^{a_{i}}=\zeta_{m_{i}}^{-a_{i}},\;\;\;\emph{for}\;\;1\leq i\leq n,\\
r_{ij}^{m_{i}}=\zeta_{m_{j}}^{-a_{ij}},\;\;\;\;\;r_{ij}^{m_{j}}=1,\;\;\;\emph{for}\;\;1\leq i<j\leq n,\nonumber \\
r_{ij}^{m_{j}}=\zeta_{m_{j}}^{a_{ij}},\;\;\;\;\;r_{ij}^{m_{i}}=1,\;\;\;\emph{for}\;\;n\geq i>j\geq 1,\nonumber\\
a_{rst}=0,\;\;\;\;\emph{for}\;\;1\leq r<s<t\leq n\nonumber.
\end{gather*}
\end{proposition}

\begin{proof}  ``$\Rightarrow$".  By the definition of $\o_{\underline{\mathbf{a}}},$ one may observe that $\o_{\underline{\mathbf{a}}}(x,y,z)=\o_{\underline{\mathbf{a}}}(x,z,y)$ for all $x,y,z \in G.$ Thus (4.3) and (4.4) can be reduced to
\begin{gather}
\R(xy, z)=\R(x , z)\R(y ,
z)\o(z,x,y),\\
\R(x , yz)=\R(x , y)\R(x , z)\frac{1}{\o(x,y,z)}.
\end{gather}

For any $1\leq i\leq n$, applying (4.5) and (4.6) iteratively, we have $\R(g_{i},g_{i}^{s})=\R(g_{i},g_{i})^{s}$ and $\R(g_{i}^{s},g_{i})=\R(g_{i},g_{i})^{s}$
for $1\leq s\leq m_{i}-1$. Then
$$1=\R(g_{i},g_{i}^{m_{i}})=\R(g_{i},g_{i})\R(g_{i},g_{i}^{m_{i}-1})
\frac{1}{\o(g_{i},g_{i},g_{i}^{m_{i}-1})}=\R(g_{i},g_{i})^{m_{i}}\frac{1}{\zeta_{m_{i}}^{a_{i}}},$$
$$1=\R(g_{i}^{m_{i}},g_{i})=\R(g_{i}^{m_{i}-1},g_{i})\R(g_{i},g_{i})\o(g_{i},g_{i}^{m_{i}-1},
g_{i})=\R(g_{i},g_{i})^{m_{i}}{\zeta_{m_{i}}^{a_{i}}}.$$
Thus $r_{ii}^{m_{i}}=\zeta_{m_{i}}^{a_{i}}=\zeta_{m_{i}}^{-a_{i}}$.

Assume $i<j.$ Applying (4.5) iteratively, one has $\R(g_{i}^{k},g_{j})=\R(g_{i},g_{j})^{k}$ for $1\leq k\leq m_{i}-1.$ Therefore,
$$1=\R(g_{i}^{m_{i}},g_{j})=\R(g_{i}^{m_{i}-1},g_{j})\R(g_{i},g_{j})\o(g_{j},g_{i}^{m_{i}-1},g_{i})
=\R(g_{i},g_{j})^{m_{i}}{\zeta_{m_{j}}^{a_{ij}}}.$$
This implies that $r_{ij}^{m_{i}}=\zeta_{m_{j}}^{-a_{ij}}$. On the other hand, it is not hard to see that $\o(g_{i},g_{j}^{s},g_{j}^{t}) \equiv 1$ for all $s,t.$ Combining this fact and (4.6), we have $\R(g_{i},g_{j}^{k})=\R(g_{i},g_{j})^{k}$ for $1\leq k\leq m_{j}$. So
$$1=\R(g_{i},g_{j}^{m_{j}})=\R(g_{i},g_{j})^{m_{j}}=r_{ij}^{m_{j}}.$$  Similarly, one has $r_{ji}^{m_{j}}=1$ and $r_{ji}^{m_{i}}=\zeta_{m_{j}}^{a_{ij}}$.

For the case $r<s<t,$ consider $\R(g_{t}g_{s},g_{r})$ and $\R(g_{s}g_{t},g_{r})$ which obviously are equal. By (4.3), we have
\begin{eqnarray*}\R(g_{t}g_{s}, g_{r})&=&\R(g_{t} , g_{r})\R(g_{s}, g_{r})\frac{\o(g_{r},g_{t},g_{s})\o(g_{t},g_{s},g_{r})}{\o(g_{t},g_{r},g_{s})}\\
&=&\R(g_{t} , g_{r})\R(g_{s}, g_{r})\zeta_{(m_{r},m_{s},m_{t})}^{-a_{rst}},\end{eqnarray*}
\begin{eqnarray*}\R(g_{s}g_{t}, g_{r})&=&\R(g_{s} , g_{r})\R(g_{t}, g_{r})\frac{\o(g_{r},g_{s},g_{t})\o(g_{s},g_{t},g_{r})}{\o(g_{s},g_{r},g_{t})}\\
&=&\R(g_{s} , g_{r})\R(g_{t}, g_{r}).\end{eqnarray*}
Therefore, $\zeta_{(m_{r},m_{s},m_{t})}^{-a_{rst}}=1$. By the choice of $\zeta_{(m_{r},m_{s},m_{t})}$
and $a_{rst},$ we arrive at $a_{rst}=0.$

The necessity is proved.

``$\Leftarrow$". Conversely, define a  map $\R: G\times G \To k^{\ast}$  by setting $$\R(g_{1}^{i_{1}}\cdots g_{n}^{i_{n}},g_{1}^{j_{1}}\cdots g_{n}^{j_{n}}):=\prod_{1\leq s, t\leq n} r_{st}^{i_{s}j_{t}}.$$ We verify that $\R$ is indeed a quasi-bicharacter. It is enough to prove that (4.3) and (4.4) hold for $\R.$ We verify only (4.3) as (4.4) can be done in the same way.

Let $x=g_{1}^{i_{1}}\cdots g_{n}^{i_{n}}, \ y=g_{1}^{j_{1}}\cdots g_{n}^{j_{n}}, \ z=g_{1}^{k_{1}}\cdots g_{n}^{k_{n}},$ then
$$\R(g_{1}^{i_{1}}\cdots g_{n}^{i_{n}}\cdot g_{1}^{j_{1}}\cdots g_{n}^{j_{n}},g_{1}^{k_{1}}\cdots g_{n}^{k_{n}})=\prod_{1\leq s, t\leq n} r_{st}^{(i_{s}+j_{s})'k_{t}},$$ where $(i_{s}+j_{s})'$ denotes the remainder of division of $i_{s}+j_{s}$ by $m_{s}.$
Consider the right-hand side of (4.3), namely, $\R(x , z)\R(y ,z)\frac{\o(z,x,y)\o(x,y,z)}{\o(x,z,y)}.$ By direct calculation, one has
$$\frac{\o(z,x,y)\o(x,y,z)}{\o(x,z,y)}=\prod_{l=1}^{n}\zeta_{m_{l}}^{a_{l}k_{l}
[\frac{i_{l}+j_{l}}{k_{l}}]}\prod_{1\leq s<t\leq n}\zeta_{m_{t}}^{a_{st}k_{t}
[\frac{i_{s}+j_{s}}{k_{s}}]}.$$
Therefore,
\begin{eqnarray*}
&&\R(x , z)\R(y ,z)\frac{\o(z,x,y)\o(x,y,z)}{\o(x,z,y)}\\
&=&\prod_{1\leq s, t\leq n} r_{st}^{(i_{s}+j_{s})k_{t}}\prod_{l=1}^{n}\zeta_{m_{l}}^{a_{l}k_{l}
[\frac{i_{l}+j_{l}}{k_{l}}]}\prod_{1\leq s<t\leq n}\zeta_{m_{t}}^{a_{st}k_{t}
[\frac{i_{s}+j_{s}}{k_{s}}]}\\
&=&\prod_{1\leq s, t\leq n} r_{st}^{((i_{s}+j_{s})'+[\frac{i_{s}+j_{s}}{m_{s}}]m_{s})k_{t}}\prod_{l=1}^{n}\zeta_{m_{l}}^{a_{l}k_{l}
[\frac{i_{l}+j_{l}}{k_{l}}]}\prod_{1\leq s<t\leq n}\zeta_{m_{t}}^{a_{st}k_{t}
[\frac{i_{s}+j_{s}}{k_{s}}]}.
\end{eqnarray*}
Note that
\begin{eqnarray*}
&&\prod_{1\leq s, t\leq n} r_{st}^{((i_{s}+j_{s})'+[\frac{i_{s}+j_{s}}{m_{s}}]m_{s})k_{t}}\\
&=&\prod_{1\leq s, t\leq n} r_{st}^{(i_{s}+j_{s})'k_{t}}\prod_{s=t=l=1}^{n} r_{ll}^{[\frac{i_{l}+j_{l}}{m_{l}}]m_{l}k_{l}}\prod_{s<t} r_{st}^{[\frac{i_{s}+j_{s}}{m_{s}}]m_{s}k_{t}}
\prod_{s>t} r_{st}^{[\frac{i_{s}+j_{s}}{m_{s}}]m_{s}k_{t}}\\
&=&\prod_{1\leq s, t\leq n} r_{st}^{(i_{s}+j_{s})'k_{t}}\prod_{s=t=l=1}^{n}r_{ll}^{[\frac{i_{l}+j_{l}}{m_{l}}]m_{l}k_{l}}\prod_{s<t} r_{st}^{[\frac{i_{s}+j_{s}}{m_{s}}]m_{s}k_{t}}\\
&=&\prod_{1\leq s, t\leq n} r_{st}^{(i_{s}+j_{s})'k_{t}}\prod_{l=1}^{n}\zeta_{m_{l}}^{-a_{l}k_{l}
[\frac{i_{l}+j_{l}}{m_{l}}]}\prod_{1\leq s<t\leq n}\zeta_{m_{t}}^{-a_{st}k_{t}
[\frac{i_{s}+j_{s}}{m_{s}}]}.
\end{eqnarray*}
This implies that $$\R(x , z)\R(y ,
z)\frac{\o(z,x,y)\o(x,y,z)}{\o(x,z,y)}=\prod_{1\leq s, t\leq n} r_{st}^{(i_{s}+j_{s})'k_{t}}=\R(xy,z).$$

The sufficiency is proved.
\end{proof}

\end{document}